\documentclass[10pt,twocolumn,twoside]{IEEEtran}

\usepackage{amsmath}
\usepackage{amssymb}
\usepackage{bbm}
\usepackage{amsthm}
\usepackage{commath}
\usepackage{graphicx}
\usepackage{tikz}
\usepackage{hyperref}
\usetikzlibrary{arrows.meta}
\usetikzlibrary{positioning}
\usetikzlibrary {shapes.misc}
\usepackage{algorithm}
\usepackage{algorithmic}
\usepackage{bm}
\usepackage[noadjust]{cite}

\usepackage{enumitem,kantlipsum}

\usetikzlibrary{calc}

\newcommand{\R}{\mathbb{R}}
\newcommand{\G}{\mathcal{G}}
\newcommand{\V}{\mathcal{V}}
\newcommand{\E}{\mathcal{E}}

\newcommand{\x}{\mathbf{x}}
\newcommand{\y}{\mathbf{y}}
\newcommand{\uu}{\mathbf{u}}

\newcommand{\NN}{\mathcal{N}}
\newcommand{\F}{\mathcal{F}}
\newcommand{\Fone}{\mathcal{F}_Z}
\newcommand{\Ftwo}{\mathcal{F}_{Z,NL}}

\newcommand{\prt}[1]{\left(#1\right)}

\newtheorem{assumption}{Assumption}
\newtheorem{theorem}{Theorem}

\newtheorem{proposition}{Proposition}
\newtheorem{lemma}{Lemma}

\newtheorem{definition}{Definition}

\DeclareMathOperator*{\argmin}{arg\,min}

\title{Identification of Nonlinear Acyclic Networks in Continuous Time from Nonzero Initial Conditions and Full Excitations}
\author{Ramachandran~Anantharaman, Renato~Vizuete, 
        Julien~M.~Hendrickx,
and~Alexandre~Mauroy
\thanks{R.~Anantharaman is with Department of Electrical Engineering, Shiv Nadar University, Noida, India. R.~Vizuete is with Université de Lorraine, CNRS, CRAN, F-54000 Nancy, France. R.~Vizuete was a FNRS Postdoctoral Researcher - CR at UCLouvain, and R. Anantharaman was a Postdoctoral researcher at University of Namur, during the elaboration of this work.  J.~M.~Hendrickx is with ICTEAM institute, UCLouvain, B-1348, Louvain-la-Neuve, Belgium. A.~Mauroy is with Department of Mathematics, Namur Research Institute for Complex Systems, University of Namur, Namur, Belgium.  
{\tt\small ramachandran.a@snu.edu.in},\\{\tt\small renato.vizuete@univ-lorraine.fr},\\{\tt\small julien.hendrickx@uclouvain.be},\\{\tt\small alexandre.mauroy@unamur.be}.}
\thanks{This work was supported by F.R.S.-FNRS via the \emph{KORNET} project, and by the \emph{SIDDARTA} Concerted Research Action (ARC) 
of the Fédération Wallonie-Bruxelles.}}

\begin{document}

\maketitle
\thispagestyle{empty}

\begin{abstract}

This paper addresses the identifiability and identification of nonlinear acyclic networks in continuous time when the dynamics are located on the edges and all the nodes are excited. First, we establish necessary and sufficient conditions for network identifiability. We show that it is necessary and sufficient to measure all the sinks to identify any tree in continuous time when the functions associated with the dynamics are analytic and satisfy $f(0)=0$, which is analogous to the discrete-time case. We then extend the result to general directed acyclic graphs (DAGs), showing that it is necessary and sufficient to measure all sinks when the dynamics are not linear (a condition that can be relaxed for trees). Next, based on these identifiability results and under the assumption of known dictionary functions, we introduce a method for the identification of trees that exploits higher order derivatives and nonzero initial conditions.
Finally, we propose a method to identify multiple parallel paths of the same length between two nodes, which allows us to identify any DAG when combined with the algorithm for the identification of trees.
Several examples are added to illustrate the results.
\end{abstract}


\maketitle

\section{Introduction}\label{sec:introduction}

Systems composed of single entities or subsystems interacting through a network are ubiquitous \cite{boccaletti2006complex,zanudo2017structure}. One of the most important problems in networked systems is the identification of the dynamics associated with the edges, which is essential for further actions like analysis, prediction or control of the evolution of the states associated with the nodes. 
Any experimental setting for identification of dynamics requires to establish its identifiability, which guarantees that it is possible to uniquely recover the dynamics. For networks, the identifiability conditions address the question: Which nodes should be excited and measured to identify the network? 
In this case, identifiability aims at deciding under which conditions the network could be theoretically identified, independently of any practical identification algorithm. If an edge is not identifiable, then it cannot be identified with any possible experiment. In contrast, if the edge is identifiable, it is possible to obtain an approximation of the dynamics with a sufficient number of experiments. When the dynamics are linear, properties of transfer functions can be exploited, and the identifiability problem has been extensively studied in \cite{weerts2018identifiability,hendrickx2019identifiability,vanwaarde2020necessary} in the full excitation case (i.e., all nodes are excited). However, the nonlinear case is naturally more challenging due to the variety of models (e.g., Volterra series, Wiener, Hammerstein, etc.) and the different possible types of functions (e.g., analytic, piecewise continuous, etc.). The problem of the identifiability of nonlinear networks  with full excitation in discrete time has been addressed in \cite{vizuete2023nonlinear,vizuete2026nonlinear}, where surprisingly, the identifiability conditions appear to be weaker than in the linear case due to the effect of nonlinearities, which allows us to distinguish the information coming from different paths.

Unfortunately, most of the identifiability conditions are obtained for discrete-time models and their extension to  continuous-time networks is not straightforward. Unlike discrete-time dynamics, a continuous-time model is based on ordinary differential equations (ODEs), where the information transmitted through successive edges is reflected in higher-order derivatives of the measured nodes. Consequently, both the identifiability analysis and the derivation of identification algorithms require tools specifically adapted to continuous-time dynamics.

Several works focusing on network identification in continuous time cover only linear dynamics \cite{liang2025frequency} or sparse networks \cite{anantharaman2025koopman}, or require the knowledge of specific variables over time \cite{verdiere2024identifiability}. However, these approaches do not directly address the problem considered in this paper: determining a minimal set of measured nodes that guarantees the identifiability of the nonlinear edge dynamics of a continuous-time network, and subsequently designing  an identification algorithm based only on the measurements at those nodes. In particular, to the best of our knowledge, a constructive identification framework directly connected to node-level identifiability conditions for nonlinear continuous-time networks has not been proposed.

In networked systems, the conception of a practical identification algorithm based on the identifiability conditions is not a trivial task, as it must take into account additional phenomena that occur in real environments \cite{rodrigues2024frequency,fonken2023local,kivits2022identification}. For instance, a simple operation such as differentiating a function used in the inference of identifiability conditions in 
\cite{vizuete2026nonlinear} cannot be implemented in real scenarios perfectly. Moreover, the presence of noise could affect the numerical stability of the identification method and have a significant impact on the network due to its propagation through  several interconnected subsystems. Finally, for large networks, it is hard to obtain an accurate identification of all the dynamics with the measurement of only a few nodes, as  is claimed by the identifiability conditions in discrete-time models \cite{vizuete2023nonlinear,vizuete2026nonlinear}. That is why it is important to design an algorithm based on the identifiability conditions which can be implemented in real nonlinear networks.

In this work, we address the problem of identification of nonlinear acyclic networks in continuous time when the dynamics are located in the edges and all the nodes are excited. Acyclic networks encompass an important family of graphs, whose identifiability has been extensively studied in the linear case \cite{mapurunga2022excitation,cheng2024identifiability}, and that we believe could be used for the analysis of more general networks that include cycles through the notion of \emph{unfolded digraphs} \cite{murphy2002dynamic,vizuete2026nonlinear}. In addition, in the full excitation case, the identifiability problem is related to determining which nodes should be measured so that there is a unique set of edge functions that leads to the information obtained through the measurements. First, we restrict the identifiability problem to the case of trees in the class of nonlinear functions satisfying $f(0)=0$, and we prove that the measurement of the sinks is necessary and sufficient for identification, including the particular case of linear functions. Then, by discarding linear functions in the network (i.e., networks with only purely nonlinear functions), we prove that the measurement of the sinks is necessary and sufficient for identifiability of general directed acyclic graphs (DAGs). Next, based on the theoretical identifiability conditions, we design algorithms for the practical identification of the nonlinear dynamics. Under the assumption that each edge function is a finite linear combination of known dictionary functions, we propose an algorithm for the identification of path graphs and trees  solely based on the information obtained with the measurement of the sinks. Further, based on the derivative of the function associated with the measurement of a node with respect to the state of the neighboring nodes, we propose a practical algorithm for the identification of multiple parallel paths of the same length between two nodes. This second algorithm guarantees the identification of any DAG when combined with the algorithm for the identification of trees. 

The rest of this paper is structured as follows. Section~\ref{sec:problem} introduces the model class and the notions of identifiability at the level of the network.  In Section~\ref{sec:trees}, we analyze the identifiability conditions for trees, and Section~\ref{sec:DAG} focuses on identifiability conditions for general DAGs. In Section~\ref{sec:algo} we propose a practical algorithm for the identification of both trees and DAGs along with numerical examples. 
Finally, conclusions and future work are presented in Section~\ref{sec:conclusions}.

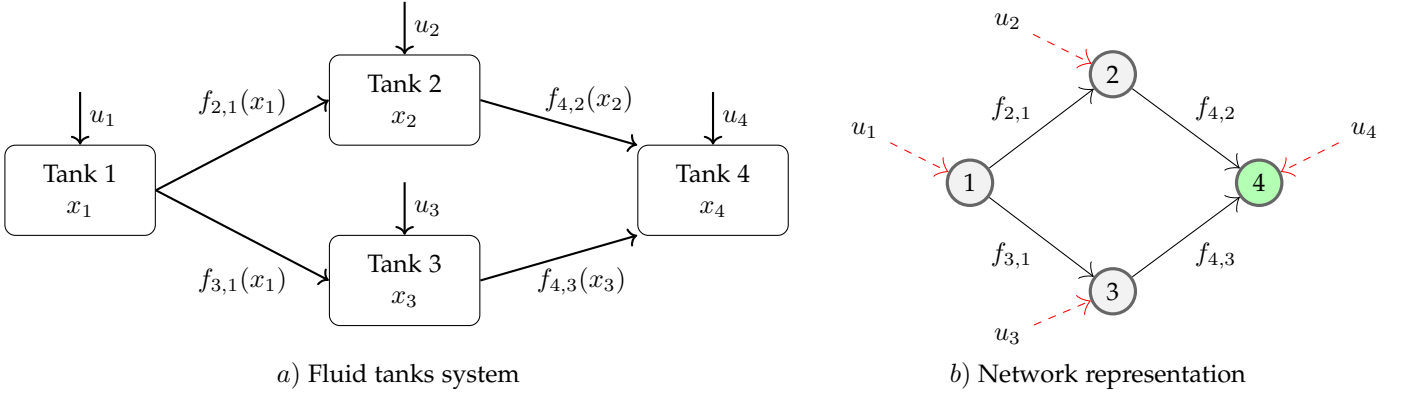
\begin{figure*}[!t]
    \centering
\begin{tikzpicture}[
    compartment/.style={draw, rounded corners, minimum width=2cm, minimum height=1.2cm, align=center},
    arrow/.style={->, thick},
    node distance=3.2cm, roundnodes/.style={circle, draw=black!60, fill=black!5, very thick, minimum size=1mm},roundnode/.style={circle, draw=white!60, fill=white!5, very thick, minimum size=1mm},roundnodes2/.style={circle, draw=black!60, fill=green!30, very thick, minimum size=1mm}
]

\node[compartment] (C1) {Tank 1 \\ $x_1$};

\node[compartment, right=2.3cm of C1, yshift=1.2cm] (C2) {Tank 2 \\ $x_2$};
\node[compartment, right=2.3cm of C1, yshift=-1.2cm] (C3) {Tank 3 \\ $x_3$};

\node[compartment, right=6.4cm of C1] (C4) {Tank 4 \\ $x_4$};

\draw[arrow] ($(C1.north)+(0,0.7)$) -- (C1.north) node[midway, right] {$u_1$};
\draw[arrow] ($(C2.north)+(0,0.7)$) -- (C2.north) node[midway, right] {$u_2$};
\draw[arrow]  ($(C3.north)+(0,0.7)$) -- (C3.north) node[midway, right] {$u_3$};
\draw[arrow] ($(C4.north)+(0,0.7)$) -- (C4.north) node[midway, right] {$u_4$};

\draw[arrow] (C1.east) -- (C2.west) node[midway, above,yshift=2.5mm] {$f_{2,1}(x_1)$};
\draw[arrow] (C1.east) -- (C3.west) node[midway, below,yshift=-3mm] {$f_{3,1}(x_1)$};

\draw[arrow] (C2.east) -- ($(C4.west)+(0,0.6)$) node[midway, above,yshift=0mm,xshift=4mm] {$f_{4,2}(x_2)$};
\draw[arrow] (C3.east) -- ($(C4.west)+(0,-0.6)$) node[midway, below,yshift=-0mm,xshift=3mm] {$f_{4,3}(x_3)$};

\node[roundnodes](node1)[right=of C1,yshift=1mm,xshift=7.3cm]{1};
\node[roundnodes](node2)[above=of node1,yshift=-2.4cm,xshift=1.9cm]{2};
\node[roundnodes2](node4)[right=3.2cm of node1]{4};
\node[roundnodes](node3)[below=of node1,yshift=2.4cm,xshift=1.9cm]{3};

\node[roundnode](u4)[above=of node4,yshift=-3.2cm,xshift=1.4cm]{$u_4$};
\node[roundnode](u3)[above=of node3,yshift=-4.5cm,xshift=-1.4cm]{$u_3$};
\node[roundnode](u2)[above=of node2,yshift=-3.2cm,xshift=-1.4cm]{$u_2$};
\node[roundnode](u1)[above=of node1,yshift=-3.2cm,xshift=-1.4cm]{$u_1$};

\draw[-{Classical TikZ Rightarrow[length=1.5mm]}] (node1) to node [left,swap,yshift=2.5mm] {$f_{2,1}$} (node2);
\draw[-{Classical TikZ Rightarrow[length=1.5mm]}] (node2) to node [right,swap,yshift=2.5mm] {$f_{4,2}$} (node4);
\draw[-{Classical TikZ Rightarrow[length=1.5mm]}] (node1) to node [left,swap,yshift=-2.5mm] {$f_{3,1}$} (node3);
\draw[-{Classical TikZ Rightarrow[length=1.5mm]}] (node3) to node [right,swap,yshift=-2.5mm] {$f_{4,3}$} (node4);

\draw[red,dashed,-{Classical TikZ Rightarrow[length=1.5mm]}] (u1) -- (node1);
\draw[red,dashed,-{Classical TikZ Rightarrow[length=1.5mm]}] (u2) -- (node2);
\draw[red,dashed,-{Classical TikZ Rightarrow[length=1.5mm]}] (u3) -- (node3);
\draw[red,dashed,-{Classical TikZ Rightarrow[length=1.5mm]}] (u4) -- (node4);

\node(graph1)[below=of node3,yshift=2.7cm,xshift=-9.5cm]{$a) $ Fluid tanks system};

\node(graph1)[below=of node3,yshift=2.7cm,xshift=-0.2cm]{$b) $ Network representation};

\end{tikzpicture}
\vspace{-7mm}
    \caption{Four well-mixed fluid tanks whose dynamics can be expressed as the model \eqref{eq:model} considered for the identification. The physical system can be abstracted as the network in panel (b) where the edges are characterized by nonlinear functions, all the nodes can be excited (red arrows), and some nodes can be measured (green).}
    \label{fig:model}
\end{figure*}

\section{Problem formulation}\label{sec:problem}

\subsection{Model class}

We consider a network modeled by a weakly connected digraph $\G=(\V,\E)$ with a set of nodes $\V=\{1,\ldots,n \}$ and a set of edges $\E\subseteq \V \times \V$, where a scalar state value $x_i\in\R$ is associated with each node $i\in\V$. The state of each node evolves according to the following dynamics:
\begin{equation}\label{eq:model}
    \dot x_i=\sum_{j\in\NN_i}f_{i,j}(x_j)+u_i,
\end{equation}
where $f_{i,j}$ is a nonlinear function associated with the edge $(i,j)$, $u_i$ is an  external signal associated only with excited nodes, and $\NN_i$ is the set of in-neighbors of $i$. This type of dynamics is a particular case of a general nonlinear network model \cite{zanudo2017structure} and is a continuous version of the discrete-time dynamics analyzed in \cite{vizuete2023nonlinear,vizuete2026nonlinear,vizuete2024partial}. Moreover, the model can be found in several applications like information distortion \cite{wang2022transmission}, nonlinear consensus algorithms \cite{liu2012discarded} and compartmental models with nonlinear flows. For instance, let us consider four well-mixed fluid tanks (compartments) connected by nonlinear valves that transport a solute between tanks, where $x_i$ is the concentration of solute in tank $i$ and $u_i$ is an external injection of solute into tank $i$. The transport across each membrane is concentration-dependent and saturating, so that the flux from tank $i$ into tank $j$ is a smooth nonlinear function of the source concentration $f_{j,i}(x_i)$ (e.g., hyperbolic tangent, logistic function). Fig.~\ref{fig:model} shows an illustration of the four tanks where the interactions can be abstracted as a nonlinear network with four nodes representing each tank.

We assume that we can measure several nodes such that the output of the system is of the form 
$$
\y=C\x,
$$
where $\x=(x_1,\ldots,x_n)^T$, $\y=(y_1,\ldots,y_n)^T$ and $C\in \{0,1\}^{p\times n} $ is a matrix with a single one and $n-1$ zeros in each row associated with the measurements of nodes. The nonzero entries of the matrix $C$ determine the set of measured nodes $\NN^m\subseteq \V$ in the system.  

The objective is to identify all the nonlinear functions $f_{i,j}$ based on the information obtained with measurements of some nodes of the network. Similarly to the common framework used for network identification \cite{weerts2018identifiability,hendrickx2019identifiability,vanwaarde2020necessary}, we make the following assumption:

\begin{assumption}\label{ass:graph_topology}
    The graph $G$ associated with the network is known, where an edge $(i,j)\in E$ implies $f_{i,j}\not\equiv  0$.
\end{assumption}

In this paper, we will consider the full excitation case. 

\begin{assumption}[Full excitation]\label{ass:full_excitation}
    All the nodes can be excited through the inputs $u_i$, which can be chosen.
\end{assumption}

Fig.~\ref{fig:model} illustrates the full excitation case where
all the nodes are excited, and some nodes are measured. Since the objective of the identifiability is to determine which nodes need to be measured to identify all the edge functions $f_{i,j}$ based on the global information $F_i$, errors during a potential
identification process are not considered for the analysis. The aim of Sections~\ref{sec:trees} and \ref{sec:DAG} is the possibility of identification, while Section~\ref{sec:algo} is focused on the practical implementation of identification algorithms. The presence of noise will be subsequently considered in Section~\ref{sec:algo}.

\subsection{Identifiability}

In the full excitation case, the goal of the identifiability conditions is to determine which nodes must be measured to identify all the nonlinear functions associated with the edges. The identifiability analysis is based on the notion of identifiability in system identification \cite{ljung1999system}, where the objective is to determine if there is a unique set of local dynamics $f_{i,j}$ that leads to a given global behavior.

Ideally, we assume that, when we measure a node $i$, we can obtain the information of the function $F_i$:
\begin{equation}\label{eq:ideal_Fi_0}
x_i(t)=F_i(\x^0 ,\uu_{[0,t]},t),    
\end{equation}
where $\x^0=\x(0)$ and $\uu_{[0,t]}=\{\uu(s):0\leq s \leq t\} $ denotes the trajectory of $\uu(t)=(u_1(t),\ldots,u_n(t))^T$ up to time $t$. For notational simplicity, throughout this work, we will denote  \eqref{eq:ideal_Fi_0} only as:
\begin{equation}\label{eq:ideal_Fi}
x_i(t)=F_i(\x^0 ,\uu(t),t).    
\end{equation}
with the understanding that the dependence on $\uu(t)$ represents the dependence on the entire input trajectory $\uu_{[0,t]}$. Based on the dynamics \eqref{eq:model}, the function $F_i$ associated with a node $i$ is of the form:
\begin{equation}\label{eq:integration_Fi}
F_i(\x^0 ,\uu(t),t)=\int_0^t\prt{\sum_{j\in\NN_i}f_{i,j}\prt{x_j(\tau)}+u_i(\tau)}d\tau+x_i^0,
\end{equation}
where each $x_j(\tau)$ can be expressed similarly as a function of the outputs of its in-neighbors, until we reach a node without incoming edges (source) $o$, whose output is of the form:
$$
F_o(\x^0 ,\uu(t),t)=\int_0^t u_o(\tau)d\tau+x_o^0.
$$
Note that \eqref{eq:ideal_Fi} is well-defined for any node, whether a node is measured or not. The function $F_i$ is the best information that we can obtain with the measurement of a node $i$ and provides the evolution of the state $x_i$ as a function of time, initial conditions and excitation signals from 0 to $t$. It is the global behavior in the context of identifiability, so that the objective is to determine whether the functions $f_{i,j}$ leading to $F_i$ are unique. In our setting, the initial conditions $\x^0$ are considered as variables whose values can be modified with appropriate excitation signals before considering the initial time for identification (this can be done ideally in any dynamical system through impulse inputs or short-duration signals). Without loss of generality, we consider the initial time $t_0=0$ for the analysis of the identifiability.
 If an edge $f_{i,j}$ is not identifiable with the information of $F_i$, it is not possible to identify it with the measurement of the node $i$. If an edge $f_{i,j}$ is identifiable with the information of $F_i$, it is possible to obtain an approximation of $f_{i,j}$ with several experiments.

The following definitions, taken from the discrete-time setting in \cite{vizuete2026nonlinear}, are also valid in continuous time.

\begin{definition}[Set of measured functions]\label{def:set_measured_functions}
  Given a set of measured nodes $\mathcal{N}^m$, the totally ordered set of measured functions  $(F(\mathcal{N}^m),\le)$ associated with $\mathcal{N}^m$ is given by:
    $$
    F(\mathcal{N}^m):=\{F_i\;|\;i\in \mathcal{N}^m\},
    $$
    with $F_i\le F_j$ if $i\le j$.
\end{definition}

The identifiability problem could be unrealistic for general functions and for this reason we restrict the problem to  a certain class of functions $\F$. In this case, we assume that the functions associated with the edges belong to $\F$ and the identifiability problem is considered only among the functions belonging to $\F$. We say that a digraph $G$ and a partially ordered set of functions $(\{ f \},\le)=\{f_{i,j}\in \F\;|\;(i,j)\in E\}$ with $f_{i,j}\le f_{m,n}$ if $i\le m$ and $j\le n$, generate $F(\mathcal{N}^m)$ if the functions $F_i\in F(\mathcal{N}^m)$ are recursively constructed from the output of the nodes by using $(\{ f \},\le)$ in the dynamical equation \eqref{eq:model}. 
The ordering introduced above is purely combinatorial and is determined by the labeling of the nodes; it does not define an analytical order between functions. Its purpose is only to preserve the correspondence between functions and their associated nodes or edges when comparing different sets of functions in the identifiability analysis. For this reason, in the rest of the paper we will refer to $\mathcal{F}(\mathcal{N}^m)$ and $\{f\}$ only as a set of measured functions and a set of functions, respectively.

The identifiability problem might not be well-posed for all the possible functions since the solution of the ODE \eqref{eq:model} does not even exist for some set of functions. Moreover, particular choices of the excitation signals $\uu(t)$ might not be integrable in \eqref{eq:integration_Fi}. For this reason, we make the following assumption:

\begin{assumption}[Well-posedness]\label{ass:well}
    The edge functions $f_{i,j}$ and the excitation signals $\uu(t)$ are such that the initial-value problem associated with \eqref{eq:model} admits a unique solution.
\end{assumption}

Next, we introduce the notion of identifiability.

\begin{definition}[Identifiability]\label{def:math_identifiability}
   Given a set of functions $\{ f \}=\{f_{i,j}\in \F\;|\;(i,j)\in E\}$ that generates $F(\NN^m)$ and any other set of functions $\{ \tilde f \}=\{\tilde f_{i,j}\in\F\;|\;(i,j)\in E\}$ that generates $\tilde F(\NN^m)$. An edge $f_{i,j}$ is identifiable in a class $\F$ if $
    F(\NN^m)=\tilde F(\NN^m)
    $ implies that $f_{i,j}=\tilde f_{i,j}$ for any $\tilde f_{i,j}$.
    A network $G$ is identifiable in a class $\F$ if 
    $
    F(\NN^m)=\tilde F(\NN^m)
    $
    implies that
    $
    \{ f \}=\{ \tilde f \}
    $ for any $\{ \tilde f \}$.
\end{definition}

In this work, we will consider analytic entire functions (i.e., for all $\hat x\in\R$, their Taylor series around $\hat x$ is convergent in a neighborhood of $\hat x$) \footnote{In \cite{vizuete2026nonlinear}, the identifiability analysis was performed for the class of twice continuously differentiable functions $C^2$. However, in the derivation of algorithms for identification in Section~\ref{sec:algo}, we use higher order derivatives of the functions, and for this reason we need to particularize the identifiability conditions for the smaller class of analytic functions.}.
In addition, since the offset of a function could make the identifiability problem unsolvable when a node has two or more in-neighbors \cite[Proposition 2]{vizuete2026nonlinear}, we will consider that $f(0)=0$ \cite{vizuete2023nonlinear}, which could be interpreted as the difference with respect to the steady state.

\begin{definition}[Class of functions $\Fone$]
Let $\Fone$ be the class of functions $f:\R\to\R$ with the following properties:
\begin{enumerate}
    \item $f$ is analytic in $\R$.
    \item $f(0)=0$.
\end{enumerate}
\end{definition}
Notice that the class $\Fone$ satisfies Assumption~\ref{ass:well}.

\section{Identifiability conditions for trees}\label{sec:trees}
Since paths play an important role for the identifiability of more complex network topologies \cite{hendrickx2019identifiability}, we will first focus on the identifiability of path graphs and hence of trees. Prior to that, we introduce the following definitions associated with digraphs.
\begin{definition}[Source and sink]
    A source is a node with no incoming edge. A sink is a node with no outgoing edge.
\end{definition}
Every DAG, including paths and trees, has at least a source and a sink.

The following proposition provides conditions for the identifiability of path graphs and characterizes two types of configurations that are sufficient to guarantee identifiability.

\begin{figure}[!ht]
	\centering
	\begin{tikzpicture}
		[
		roundnodes/.style={circle, draw=black!60, fill=black!5, very thick, minimum size=7mm},roundnode/.style={circle, draw=white!60, fill=white!5, very thick, minimum size=1mm},roundnodes2/.style={circle, draw=black!60, fill=green!30, very thick, minimum size=7mm}
        ,roundnodes3/.style={circle, draw=white!60, fill=white!30, very thick, minimum size=7mm}
		]
		
		\node[roundnodes](node1){1};
		\node[roundnodes](node2)[right=1.25cm of node1]{2};
        \node[roundnodes3](node5)[right=1.25cm of node2]{$\cdots$};
		\node[roundnodes](node3)[right=1.25cm of node5]{\footnotesize{$\!\!n\!\!-\!\!1\!\!$}};
		\node[roundnodes2](node4)[right=1.25cm of node3]{$n$};

		\draw[-{Classical TikZ Rightarrow[length=1.5mm]}] (node2) to node [above,swap] {} (node5);
        \draw[-{Classical TikZ Rightarrow[length=1.5mm]}] (node5) to node [above,swap] {} (node3);
		\draw[-{Classical TikZ Rightarrow[length=1.5mm]}] (node1) to node [above,swap] {$f_{2,1}$} (node2);
		\draw[-{Classical TikZ Rightarrow[length=1.5mm]}] (node3) to node [above,swap] {$f_{n,n-1}$} (node4);

	\end{tikzpicture}
	
	\caption{Path graph with $n$ nodes.}
	\label{fig:path_graph_n}
\end{figure}

\begin{proposition}[Path graphs]\label{prop:paths}
Under Assumptions \ref{ass:graph_topology}-\ref{ass:well}, a path graph is identifiable in the class $\Fone$, if and only if the sink is measured. Moreover, the identifiability remains true even if we use zero initial conditions and constant nonzero excitation signals, or use multiple initial conditions and zero excitation signals.
\end{proposition}

Before presenting the proof of Proposition~\ref{prop:paths}, we recall a
technical result that will be used in the proof.
\begin{lemma}[Lemma 4 \cite{vizuete2023nonlinear}]\label{lemma:identification_one_function}
Let $f:\R\to\R$ and $g,\tilde g:\R^m\to \R$ be three nonzero analytic functions, with $g(0)=\tilde{g}(0)=0$. If, for all $x\in\R$, $y\in\R^m$, the functions $f$, $g$ and $\tilde g$ satisfy
$$
f(x+g(y_1,\ldots,y_m))=f(x+\tilde{g}(y_1,\ldots,y_m)),
$$
then either $g=\tilde g$ or $f$ is constant.
\end{lemma}

\begin{myproof}{Proposition~\ref{prop:paths}}
    Let us assume a path graph with $n\ge 2$ nodes where we use the labeling 1 for the source and $n$ for the sink as shown in Fig.~\ref{fig:path_graph_n}. The measurement of the sink $n$ is necessary since the last edge $f_{n,n-1}$ only affects the sink in the graph. 
    The measurement of the sink provides an output of the form
    \begin{align}    x_n(t)&=F_n(\x^0,\uu,t)\label{eq:sink_path}\\
    &=\int_0^t f_{n,n-1}(x_{n-1}(\tau))d\tau+u_nt+x_n^0.    
    \end{align}
    Notice that if $x_{n-1}(t)$ is an analytic function, then $\int_0^t f_{n,n-1}(x_{n-1}(\tau))d\tau$ is also analytic \cite[Proposition 2.2.3]{krantz2002primer} and hence, $F_n(\x^0,\uu,t)$ is also analytic in the variable $t$. Since $u_1(t)=u_1$, the output of the source $x_1(t)=u_1t+x_1^0$ is analytic, and hence $x_2(t)$ is also analytic. Applying a similar reasoning, we have that
    all the functions $F_i(\x^0,\uu,t)$ are analytic for all $i\in\V$.
    Let us assume that there exists a set of functions $\{ \tilde f_{i,j}\}$ that generate $\tilde F_n(u,t)$ such that $F_n=\tilde F_n$. Then,
    \begin{equation}\label{eq:before_cases}
    \int_0^t f_{n,n-1}(x_{n-1}(\tau))d\tau=\int_0^t \tilde f_{n,n-1}(\tilde x_{n-1}(\tau))d\tau.    
    \end{equation}

    \noindent\textbf{1) Identifiability with zero initial conditions and constant nonzero excitation signals}

    \noindent In \eqref{eq:before_cases}, if we set to zero all the initial conditions and all the input signals except for $u_{n-1}$, we obtain
    $$
    \int_0^t f_{n,n-1}(u_{n-1}\tau )d\tau=\int_0^t \tilde f_{n,n-1}(u_{n-1}\tau)d\tau,
    $$
    which implies $f_{n,n-1}=\tilde f_{n,n-1}$ from the Fundamental Theorem of Calculus.  The identifiability of $f_{n,n-1}$ implies that
    $$
    f_{n,n-1}(F_{n-1}(\x^0,\uu,t))=f_{n,n-1}(\tilde F_{n-1}(\x^0,\uu,t)),
    $$
    which is equivalent to
    \begin{multline*}
    f_{n,n-1}\prt{\int_0^t f_{n-1,n-2}(x_{n-2}(\tau))d\tau+u_{n-1}t+x_{n-1}^0}=\\f_{n,n-1}\prt{\int_0^t \tilde f_{n-1,n-2}(x_{n-2}(\tau))d\tau+u_{n-1}t+x_{n-1}^0};    
    \end{multline*}
    \begin{multline}\label{eq:function_g}
        f_{n,n-1}\prt{g_{n,n-1}(\x^0,\uu,t)+u_{n-1}t+x_{n-1}^0}=\\f_{n,n-1}\prt{\tilde g_{n,n-1}(\x^0,\uu,t)+u_{n-1}t+x_{n-1}^0},
    \end{multline}
    where $g_{n,n-1}(\x^0,\uu,t)=\int_0^t f_{n-1,n-2}(x_{n-2}(\tau))d\tau$ and $\tilde g_{n,n-1}(\x^0,\uu,t)=\int_0^t \tilde f_{n-1,n-2}(x_{n-2}(\tau))d\tau$.
    Let us assume that there exists a nonzero $(\x^{0,*},\uu^*,\tau^*)\in\R^{2n+1}$
 such that $g_{n,n-1}(\x^{0,*},\uu^*,\tau^*)=a$ and $\tilde g_{n,n-1}(\x^{0,*},\uu^*,\tau^*)=b$ with $a\neq b$. Then we get
 \begin{equation}\label{eq:u_and_x}
    f_{n,n-1}(a+u_{n-1}\tau^*+x_{n-1}^0)=f_{n,n-1}(b+u_{n-1}\tau^*+x_{n-1}^0). 
 \end{equation}
 If we set $x_{n-1}^0=0$, we have
 $$
    f_{n,n-1}(a+u_{n-1}\tau^*)=f_{n,n-1}(b+u_{n-1}\tau^*) \quad \text{for all } u_{n-1}\in\R,
    $$ 
    which is equivalent to
    \begin{equation}\label{eq:z}
    f_{n,n-1}(z)=f_{n,n-1}(z+b-a) \quad \text{for all } z\in\R.    
    \end{equation}
    But according to Lemma~\ref{lemma:identification_one_function}, this is not possible since the function $f_{n,n-1}$ is analytic, so that we get a contradiction and $g_{n,n-1}(\x^0,\uu,t)=\tilde g_{n,n-1}(\x^0,\uu,t)$. Then, we obtain
    $$
    \int_0^t f_{n-1,n-2}(x_{n-2}(\tau))d\tau=\int_0^t \tilde f_{n-1,n-2}(\tilde x_{n-2}(\tau))d\tau,
    $$
    and we prove that $f_{n-1,n-2}=\tilde f_{n-1,n-2}$ by using a similar approach. Notice that we can use the same procedure until we reach the source such that all the path graph can be identified.

    \noindent\textbf{2) Identifiability with multiple initial conditions and zero excitation signals}

    \noindent To prove the second case, we just make two modifications in the proof of case 1). First, in eq.~\eqref{eq:before_cases}, if we set to zero all the excitation signals and initial conditions, except for $x_{n-1}^0$, we get
    $$
     f_{n,n-1}(x_{n-1}^0 )= \tilde f_{n,n-1}(x_{n-1}^0), \text{ for all } x_{n-1}^0\in \R,
    $$
    which yields $f_{n,n-1}=\tilde f_{n,n-1}$. Second, in eq.~\eqref{eq:u_and_x}, if we set $u_{n-1}=0$, we obtain
   $$
    f_{n,n-1}(a+x_{n-1}^0)=f_{n,n-1}(b+x_{n-1}^0) \quad \text{for all } x_{n-1}^0\in\R,
    $$
    which is equivalent to \eqref{eq:z}. The rest of the proof is similar to the case 1).
\end{myproof}

Proposition~\ref{prop:paths} provides two different approaches for potential identification algorithms where we could use nonzero constant excitation signals, or multiple initial conditions and zero excitation signals. This is due to the fact that the mapping~\eqref{eq:ideal_Fi} depends on two different types of variables: initial conditions and excitation signals.
Based on Proposition~\ref{prop:paths}, we derive conditions for the identifiability of trees and provide two types of configurations that are sufficient to guarantee identifiability.

\begin{theorem}[Trees]
\label{th:Trees}
Under Assumptions \ref{ass:graph_topology}-\ref{ass:well}, a tree is identifiable in the class $\Fone$ if and only if all the sinks are measured. Moreover, the identifiability remains true even if we use zero initial conditions and constant nonzero excitation signals, or use multiple initial conditions and zero excitation signals.
\end{theorem}
\begin{proof}
The measurement of all the sinks is necessary since their incoming edges only affect them. Let us select a sink $s$ of the tree. If, at each node $i$, we set to zero the inputs and initial conditions of all the branches that arrive to $i$ except one, in this particular setting, the function $F_s$ associated with the measurement of the sink $s$:
$$
x_s(t)=F_s(\x^0,\uu,t),
$$
can be linked to a single path graph, so that we can use Proposition~\ref{prop:paths} to guarantee the identifiability of all the edges in that path. By repeating the same procedure for the other edges and sinks, we prove the identifiability of all the tree.  
\end{proof}

\section{Identifiability of Directed acyclic graphs}\label{sec:DAG}

Unlike the case of trees, the identifiability of general DAGs in the linear case might require the measurement of the sinks and other additional nodes in discrete time \cite{hendrickx2019identifiability}. This is due to the fact that in DAGs,  there can exist several paths connecting two nodes that do not allow us to distinguish the information coming from the different paths. 
For instance, let us consider the DAG in Fig.~\ref{fig:model} where we assume that the functions $f_{4,2}$ and $f_{4,3}$ have already been identified.
Ideally, the measurement of node 4 provides the information:
$$
x_4(t)=\int_0^tf_{4,2}(x_2(\tau))d\tau+\int_0^tf_{4,3}(x_3(\tau))d\tau+u_4t+x_4^0.
$$
If the functions $f_{4,2}$ and $f_{4,3}$ are linear, i.e. of the form $f_{4,2}(x)=k_{42}x$ and $f_{4,3}(x)=k_{43}x$, we would have
\begin{align}
    x_4(t)\!&=k_{42}\int_0^tx_2(\tau)d\tau + k_{43}\int_0^tx_3(\tau)d\tau+u_4t+x_4^0\nonumber\\
    &=k_{42}\int_0^t\prt{\int_0^\tau f_{2,1}(x_1(\theta))d\theta+u_2\tau+x_2^0}d\tau+\nonumber\\
    &\hspace{4.5mm}k_{43}\int_0^t\prt{\int_0^\tau f_{3,1}(x_1(\theta))d\theta\!+\!u_3\tau\!+\!x_3^0}d\tau\!+\!u_4t\!+\!x_4^0.\label{eq:measurement_4}
\end{align}
Notice that the functions $\tilde f_{2,1}=f_{2,1}+\frac{1}{k_{42}}\psi$ and $\tilde f_{3,1}=f_{3,1}-\frac{1}{k_{43}}\psi$ also satisfy \eqref{eq:measurement_4} for any function $\psi$ in the class $\Fone$,  so that the edges $f_{2,1}$ and $f_{3,1}$ are not unique and cannot be identified. This is because of the superposition principle of linear functions, which mixes the information coming from different paths \cite{vizuete2026nonlinear}.
Hence, we would need to measure an additional node to identify the network. Nevertheless, if the functions $f_{4,2}$ and $f_{4,3}$ are not linear, the derivatives with respect to $x_2$ and $x_3$, respectively, would allow us to isolate the information coming from each path \cite{vizuete2026nonlinear}.
For this reason, in order to relax these conditions, we will focus on the identifiability of DAGs in the class of pure nonlinear functions \cite{vizuete2023nonlinear}.

\begin{definition}[Class of functions $\Ftwo$]
Let $\Ftwo$ be the class of functions $f:\R\to\R$ with the following properties:
\begin{enumerate}
    \item $f$ is analytic in $\R$.
    \item $f(0)=0$.
    \item The associated Taylor series $f(x)=\sum_{n=1}^\infty a_nx^n$ contains at least one coefficient $a_n\neq 0$ with $n>1$.    
\end{enumerate}
\end{definition}

In eq.~\eqref{eq:function_g} corresponding to the proof of Proposition~\ref{prop:paths}, the identifiability problem requires us to show that the function $g$ is unique. Using the same terminology as in \cite[Definition 6]{vizuete2026nonlinear}, we will say that the function $g$ is identifiable if it is the only possible function that leads to the global behavior determined by the measured functions $F_i$. In addition, we provide an extension of \cite[Lemma 4]{vizuete2026nonlinear} to the continuous-time dynamics.
\begin{lemma}[Removal of a node]\label{lemma:removal_node}
    Consider a DAG $G$ and a node $j$ with only one path to its out-neighbor $i$. Let us assume that the nodes $i$ and $j$ have been measured and the edge $f_{i,j}$ is identifiable. For constant excitation signals, the edges $f_{i,\ell}$ and the functions $g_\ell$ for $\ell\in\NN_i$ are identifiable in $G$ if they are identifiable in the induced subgraph   
    $G_{V\setminus\{ j\}}$.
\end{lemma}
\begin{proof}
    The proof is deferred to Appendix~\ref{app:proof}.
\end{proof}

Analogous to Proposition~\ref{prop:paths}, the following theorem provides identifiability conditions for DAGs and characterizes two types of configurations that are sufficient to guarantee identifiability.

\begin{theorem}[Directed acyclic graphs]\label{prop:DAGs}
Under Assumptions \ref{ass:graph_topology}-\ref{ass:well}, a DAG is identifiable in the class $\Ftwo$ if and only if all the sinks are measured. Moreover, the identifiability remains true even if we use zero initial conditions and constant nonzero excitation signals, or use multiple initial conditions and zero excitation signals.
\end{theorem}
\begin{proof}
Let us assume an arbitrary DAG $G$.
For constant excitation signals, the measurement of a sink $i$ provides an output of the form
\begin{align*}
x_i(t)&=F_i(\x^0,\uu,t)\\
    &=\int_{0}^t\sum_{j\in \NN_i}f_{i,j}(x_j(\tau))d\tau+u_it+x_i^0\\
    &=\sum_{j\in \NN_i}\prt{\int_{0}^t f_{i,j}(x_j(\tau))d\tau}+u_it+x_i^0.
\end{align*}
Let us assume that there exists a set of functions $\{ \tilde f_{i,j} \}$ that generate $\tilde F_i(\x^0,\uu,t)$ such that $F_i=\tilde F_i$. Then, we get
\begin{equation}\label{eq:equality_sum_functions}
\sum_{j\in \NN_i}\int_{0}^t f_{i,j}(x_j(\tau))d\tau=\sum_{j\in \NN_i}\int_{0}^t \tilde f_{i,j}(\tilde x_j(\tau))d\tau.    
\end{equation}
\textbf{1) Identifiability with zero initial conditions and constant nonzero excitation signals} \newline 
\noindent In any DAG, there is a node $\ell$ with only one path to any node $i$ \cite{vizuete2026nonlinear}. If we set to zero all the initial conditions and all the excitation signals except $u_\ell$, \eqref{eq:equality_sum_functions} becomes
$$
\int_{0}^t f_{i,\ell}(u_\ell \tau)d\tau=\int_{0}^t \tilde f_{i,\ell}(u_\ell \tau)d\tau,  
$$
which implies $f_{i,\ell}=\tilde f_{i,\ell}$ from the Fundamental Theorem of Calculus. If we take the derivative with respect to $u_\ell$ in  \eqref{eq:equality_sum_functions}, we get
\begin{equation}
\frac{d}{du_\ell}\prt{\int_{0}^t f_{i,\ell}(x_\ell(\tau))d\tau}
    =\frac{d}{du_\ell}\prt{\int_{0}^t  f_{i,\ell}(\tilde x_\ell(\tau))d\tau}. 
\end{equation}
Since the function $f_{i,\ell}$ is analytic with respect to all its arguments, we can use the Leibniz integral rule to exchange the integral and the derivative.
By using  $x_\ell(t)=g_\ell(\x^0,\uu,t)+u_\ell t+x_\ell^0$ we have
\begin{multline*}
\int_{0}^t\frac{\partial}{\partial u_\ell}\prt{ f_{i,\ell}(g_\ell(\x^0,\uu,\tau)+u_\ell\tau+x_\ell^0)}d\tau\\
    =\int_{0}^t\frac{\partial}{\partial u_\ell}\prt{  f_{i,\ell}(\tilde g_\ell(\x^0,\uu,\tau)+u_\ell\tau+x_\ell^0)}d\tau,    
\end{multline*}
which yields
\begin{multline}\label{eq:temporal_with_tau}
\int_{0}^t f'_{i,\ell}(g_\ell(\x^0,\uu,\tau)+u_\ell\tau+x_\ell^0)\tau d\tau
    =\\
    \int_{0}^t  f'_{i,\ell}(\tilde g_\ell(\x^0,\uu,\tau)+u_\ell\tau+x_\ell^0)\tau d\tau.     
\end{multline}
This implies
$$
f'_{i,\ell}(g_\ell(\x^0,\uu,\tau)+u_\ell\tau+x_\ell^0)=f'_{i,\ell}(\tilde g_\ell(\x^0,\uu,\tau)+u_\ell\tau+x_\ell^0).
$$
Since the functions $f'_{i,\ell}$, $g_\ell$ and $\tilde g_\ell$ are analytic, according to Lemma~\ref{lemma:identification_one_function}, either $g_\ell=\tilde g_\ell$ or $f'_{i,\ell}$ is constant. However, since each function $f_{i,j}$ is in $\Ftwo$, the derivative $f'_{i,\ell}$ cannot be constant, which implies $g_\ell=\tilde g_\ell$.
    Then, since we showed above that $f_{i,\ell}=\tilde f_{i,\ell}$, \eqref{eq:equality_sum_functions} becomes
    \begin{equation}\label{eq:new_identifiability_problem}
         \sum_{j\in \NN_i\setminus \{\ell \}}\int_{0}^t f_{i,j}(x_j(\tau))d\tau=\sum_{j\in \NN_i\setminus \{\ell \}}\int_{0}^t \tilde f_{i,j}(\tilde x_j(\tau))d\tau.  
    \end{equation}
    According to Lemma~\ref{lemma:removal_node}, the identifiability problem \eqref{eq:new_identifiability_problem} can be analyzed in the induced DAG $G_{\V\setminus \{ \ell\}}$, where there must be another in-neighbor of $i$ that has only one path to $i$. By using the same approach, we can identify all the incoming edges of $i$ and since $x_j(t)=\tilde x_j(t)$ for $j\in\NN_i$, we can use the same procedure for the other nodes until we reach the sources of the DAG that can be reached from the sink $i$. By measuring the other sinks, we guarantee the identifiability of all the network.

    \noindent \textbf{2) Identifiability with multiple initial conditions and zero excitation signals}

    \noindent For the proof of the second case, we make two modifications in the proof of case 1). First, in eq.~\eqref{eq:equality_sum_functions}, if we set to zero all the excitation signals and initial conditions, except for $x_\ell^0$, we obtain
    $$
    f_{i,\ell}(x_\ell^0)=\tilde f_{i,\ell}(x_\ell^0), \text{ for all } x_{\ell}^0\in \R,
    $$
    which guarantees $f_{i,\ell}=\tilde f_{i,\ell}$. Second, in eq.~\eqref{eq:equality_sum_functions}, we take the derivative with respect to $x_\ell^0$, which yields
    \begin{multline*}
\int_{0}^t f'_{i,\ell}(g_\ell(\x^0,\uu,\tau)+u_\ell\tau+x_\ell^0)d\tau
    =\\
    \int_{0}^t  f'_{i,\ell}(\tilde g_\ell(\x^0,\uu,\tau)+u_\ell\tau+x_\ell^0) d\tau,     
\end{multline*}instead of \eqref{eq:temporal_with_tau}. The rest of the proof is similar to the \linebreak case 1).
\end{proof}

Similarly to the identifiability conditions in the discrete-time case \cite{vizuete2026nonlinear}, the nonlinearity of the functions allows us to distinguish the information coming from different paths to the same node. This characteristic of the class $\Ftwo$ allows us to obtain weaker conditions compared to the linear case and will play an important role in the conception of an algorithm for identification.

\section{Algorithms for identification}
\label{sec:algo}
Given that the identifiability conditions for trees and DAGs have been established in the previous sections, we proceed to develop a practical algorithm for identification. Even if the identifiability conditions of Sections~\ref{sec:trees} and \ref{sec:DAG} were derived for continuous-time dynamics, the information collected via an experimental setting comes from discrete measurements at specific time instants. For this reason, it is important to design an algorithm that is not only based on the theoretical identifiability conditions, but also takes into account the discrete nature of the measurements. Moreover, we will focus on the specific setting of multiple initial conditions (i.e. case 2) in Proposition \ref{prop:paths} and Theorem~\ref{prop:DAGs}). This setting requires the application of impulse-type signals  to modify the initial conditions of the nodes so that it might not be applicable for some types of systems in real life. However, this does not exclude the possibility that further algorithms could be designed to use directly constant inputs for other types of systems. 

Similarly to \cite{mauroy2020koopman}, we make the following assumption on the functions $f_{i,j}$ in the network.

\begin{assumption}\label{ass:dic_fun}
	Each function $f_{i,j}$ in the network $\G$ is a finite linear combination of known dictionary functions
	
	\vspace{-3mm}
	\begin{align}
		\label{eq:dic_fun}
		f_{i,j} = \sum_{\ell=1}^{L_{ij}} {\alpha}^*_{ij\ell} \phi_{ij\ell},
	\end{align}
	so that the identification of the function $f_{i,j}$ is equivalent to the identification of the coefficients ${\alpha}^*_{ij\ell}$. 
\end{assumption}

Additionally, in the experimental procedures, we will use the higher-order time derivatives of measured states, which are supposed to be known from the ideal information corresponding to $F_i$ in \eqref{eq:ideal_Fi}. The practical algorithm we propose is a sequential multi-step approach, where a few edge functions are identified at each step. The  functions identified from the previous steps are used in the subsequent steps.

\subsection{Algorithm for identification of trees}
\label{sec:Algorithm}
In order to formulate an algorithm for the identification of trees, we consider for simplicity a path graph of $4$ nodes as in Fig. \ref{fig:path_graph_4}. The method can be easily extended to a path graph of any length. From Theorem~\ref{th:Trees}, it is necessary and sufficient to measure only node 4 (sink) of the graph to identify the dynamics on the edges.

\begin{figure}[!ht]
	\centering
	\begin{tikzpicture}
		[
		roundnodes/.style={circle, draw=black!60, fill=black!5, very thick, minimum size=1mm},roundnode/.style={circle, draw=white!60, fill=white!5, very thick, minimum size=1mm},roundnodes2/.style={circle, draw=black!60, fill=green!30, very thick, minimum size=1mm}
		]
		
		\node[roundnodes](node1){1};
		\node[roundnodes](node2)[right=1.7cm of node1]{2};
		\node[roundnodes](node3)[right=1.7cm of node2]{3};
		\node[roundnodes2](node4)[right=1.7cm of node3]{4};

		\draw[-{Classical TikZ Rightarrow[length=1.5mm]}] (node2) to node [above,swap] {$f_{3,2}$} (node3);
		\draw[-{Classical TikZ Rightarrow[length=1.5mm]}] (node1) to node [above,swap] {$f_{2,1}$} (node2);
		\draw[-{Classical TikZ Rightarrow[length=1.5mm]}] (node3) to node [above,swap] {$f_{4,3}$} (node4);

	\end{tikzpicture}
	
	\caption{Path graph with 4 nodes.}
	\label{fig:path_graph_4}
\end{figure}
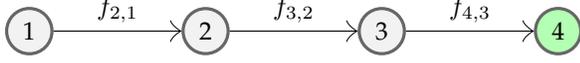

The dynamics associated with the path graph in Fig. \ref{fig:path_graph_4} are given as: 

\begin{align}
	\label{eq:Path_4nodes}
	\begin{aligned}
		\dot{x}_4 &= f_{4,3}(x_3) + u_4 \\
		\dot{x}_3 &= f_{3,2}(x_2) + u_3 \\ 
		\dot{x}_2 &= f_{2,1}(x_1) + u_2 \\
		\dot{x}_1 &= u_1.
	\end{aligned}
\end{align}

In this multi-step procedure, starting from the edge connected to the sink, each step identifies one successive edge away from the sink. Additionally, during the identification of $f_{i,j}$, all nodes that have a path to node $j$ are initialized with zero initial conditions and are not excited by external inputs. Since we measure the sink,  the solution $x_4(t)$ is known and we assume that its higher-order time derivatives are also known, or computed with sufficient accuracy. 
\begin{enumerate}[wide, labelwidth=!, labelindent=0pt]
	\item \emph{Identification of $f_{4,3}$}: We initialize the system with $K$ different initial conditions, $x_0^{(k)} := (x_{10}^{(k)},\dots,x_{40}^{(k)})^T$, with $x_{30}^{(k)} \neq 0$, $x_{j0}^{(k)} = 0$, $j \neq 3$, $k = 1,\dots,K$, and zero inputs $\uu(t) = 0^T$. 
	Evaluating $\dot{x}_4$ at time $t=0$, we obtain
	\[
	\dot{x}_4^{(k)}(0) = f_{4,3} (x_{3}^{(k)}(0)) = \sum_{\ell} \alpha_{43\ell} \phi_{43\ell}(x_{30}^{(k)}). 
	\]
	Assuming $K > L_{43}$, and the knowledge of the dictionary functions $\phi_{43\ell}$, we can recover the coefficients $\alpha^*_{43\ell}$ by formulating the regression problem 
	\[
	{\alpha}^*_{43\ell} = \argmin_{{\alpha}_{43\ell}} \sum_{k=1}^{K} \left(\dot{x}_4^{(k)}(0) - \sum_{\ell} \alpha_{43\ell} \phi_{43\ell}(x_{30}^{(k)})\right)^2.
	\]
	\item \emph{Identification of $f_{3,2}$}: 
	Once $f_{4,3}$ is identified, we initialize the system with multiple initial conditions $x_{20}^{(k)}, x_{30}^{(k)} \neq 0, k = 1,\dots,K$, and $\uu(t) = 0^T$. We have 
	$$
	\!\ddot{x}_4 \!=\! \frac{df_{4,3}(x_3)}{dt} \!=\!  \frac{\partial\! f_{4,3}}{\partial x_3} \frac{dx_3}{dt} 
	\!=\! \frac{\partial\! f_{4,3}}{\partial x_3} \big( \sum_\ell \alpha_{32\ell} (x_2) \big).
	$$
	Since $f_{4,3}$ is known, we can choose $x_{30}^{(k)}$ such that $\frac{\partial f_{4,3}}{\partial x_3}|_{x_{30}^{(k)}} \neq 0$. Choosing $K > L_{32}$, and the measurements of $\ddot{x}_4^{(k)}(0)$, we can solve for the parameters $\alpha^*_{32\ell}$ by formulating the regression problem 
	\[
	{\alpha}^*_{32\ell} = \argmin_{{\alpha}_{32\ell}} \sum_{k=1}^{K} \left( \frac{\ddot{x}_4^{(k)}(0)}{\frac{\partial f_{4,3}}{\partial x_3}|_{x_{30}^{(k)}}} - \sum_{\ell} \alpha_{32\ell} \phi_{32\ell}(x_{20}^{(k)})\right)^2.
	\]
	\item \emph{Identification of $f_{2,1}$}: From the knowledge of $f_{4,3}$ and $f_{3,2}$, we proceed to identify the last edge $f_{2,1}$. We initialize $x_{10}^{(k)},x_{20}^{(k)},x_{30}^{(k)} \neq 0$ and $\uu(t) = 0^T$. The third derivative of $x_4$ is computed as 
	\[
	\!\dddot{x}_4\! =\! \frac{\partial^2 \!f_{4,3}}{\partial x_3^2} (f_{3,2}(x_2))^2\! + \frac{\partial\! f_{4,3}}{\partial x_3}  \frac{\partial \! f_{3,2}}{\partial x_2} \big(\sum_\ell \alpha_{21\ell}\phi_{21\ell}(x_1)\big).
	\]
	From the knowledge of $f_{4,3},f_{3,2},$ we choose the initial conditions $x_{30}^{(k)}$, $x_{20}^{(k)}$, such that $\frac{\partial f_{4,3}}{\partial x_3}|_{x_{30}^{(k)}} \neq 0$ and $\frac{\partial f_{3,2}}{\partial x_2}|_{x_{20}^{(k)}} \neq 0$ for all $k$, and given the measurements of $\dddot{x}_4^{(k)}(0)$, we can solve for the parameters $\alpha^*_{21\ell}$ by formulating the regression problem
	\begin{align*}
		{\alpha}^*_{21\ell} = \argmin_{\alpha_{21\ell}} \sum_{k=1}^{K} &\left( \frac{\dddot{x}_4^{(k)}(0) - \frac{\partial^2 f_{4,3}}{\partial x_3^2}(f_{3,2}(x_{20}^{(k)}))^2}{\frac{\partial f_{4,3}}{\partial x_3}|_{x_{30}^{(k)}} \frac{\partial f_{3,2}}{\partial x_2}|_{x_{20}^{(k)}}}  \right.\\  & \quad -\left. \sum_{\ell} \alpha_{21\ell} \phi_{21\ell}(x_{10}^{(k)})\right)^2.
	\end{align*}
\end{enumerate}

Since our theoretical identifiability results guarantee uniqueness of the functions by using multiple initial conditions and zero excitation signals,
the regression problems are well-posed, i.e.  admit a unique solution, provided that the number $K$ of initial conditions is large enough, and the initial conditions are independent enough. In practice, the choice of $K$ depends on the numbers $L_{ij}$ of dictionary functions in \eqref{eq:dic_fun}. Additionally, the nonzero components in each initial condition $x_0^{k}$ are chosen such that the regression problem is well-posed. In general, the above procedure can be extended to a path graph of any length, and the measurement of the derivatives up to order $n-1$ at the sink is essential for identification of the edges in a path of $n$ nodes. This is consistent with the discrete-time dynamics, where the effect of an input at distance $n$ from the sink is reflected in the sink after a delay of $n$ time instants. The procedure for the identification of a path graph is summarized in Algorithm \ref{alg:path_graph}.

\begin{algorithm}
	\caption{Identification of path graph with $n$ nodes}
	\label{alg:sec_step}
	\begin{algorithmic}[1] 
		\renewcommand{\algorithmicrequire}{\textbf{Input:}}
		\renewcommand{\algorithmicensure}{\textbf{Output:}}
		\REQUIRE Measurement of the sink $F_n$, set of dictionary functions $\phi_{ii-1\ell}$ of cardinality $L_{i,i-1}$, where $i=2,\dots,n$.
		\ENSURE The parameters $\alpha^*_{ii-1\ell}$ corresponding to the function $f_{i,i-1}$ as in  equation (\ref{eq:dic_fun}).
		\STATE $i \gets n$
		\STATE $K = \max(L_{ii-1})$
		\WHILE{$i > 1$}
		\STATE Set $x_{j}^{(k)}(0) = 0\ \forall\ 1< j < i-2$, $k = 1,\dots,K$.
		\STATE Set $\uu(t) = 0$
		\IF{$i = n$}
		\STATE Choose different initial conditions $x_{n-1}^{(k)}(0) \neq 0$ for all $k = 1,\dots,K$.
		\ELSE
		\STATE Choose $x_{j-1}^{(k)}(0) \neq 0$ such that $\frac{\partial f_{j,j-1}}{\partial x_{j-1}}|_{x_{j-1}^{(k)}(0)} \neq 0$, $k = 1,\dots,K$, and $j=i-1,\dots,n-1$
		\ENDIF
		\STATE Using the $(n-i+1)^{th}$ derivative of $x_n$, and the functions $f_{j+1,j}$ for $j=i,\dots,n-1$, and the dictionary functions for $f_{i,i-1}$ formulate a regression problem and solve for the coefficients $\alpha^*$. 
		\STATE $i \gets i-1$.
		\ENDWHILE
	\end{algorithmic}
	\label{alg:path_graph}
\end{algorithm}

Since a specific path in a tree can be isolated by setting to zero the excitation signals of the nodes that do not belong to this path, the method remains valid for the identification of trees. 

\subsection{Algorithm for identification of DAGs}
\label{sec:Algorithm_DAG}
\begin{figure}[!ht]
    \centering
    \begin{tikzpicture}
    [
roundnodes/.style={circle, draw=black!60, fill=black!5, very thick, minimum size=1mm},roundnode/.style={circle, draw=white!60, fill=white!5, very thick, minimum size=1mm},roundnodes2/.style={circle, draw=black!60, fill=green!30, very thick, minimum size=1mm}
]
\node[roundnodes](node1){1};
\node[roundnodes](node2)[above=of node1,yshift=1mm,xshift=1.9cm]{2};
\node[roundnodes2](node3)[right=3.2cm of node1]{3};

\draw[-{Classical TikZ Rightarrow[length=1.5mm]}] (node1) to node [left,swap,yshift=2.5mm] {$f_{2,1}$} (node2);
\draw[-{Classical TikZ Rightarrow[length=1.5mm]}] (node2) to node [right,swap,yshift=2.5mm] {$f_{3,2}$} (node3);
\draw[-{Classical TikZ Rightarrow[length=1.5mm]}] (node1) to node [below,swap,yshift=0mm] {$f_{3,1}$} (node3);

\end{tikzpicture}
    \caption{DAG whose nonlinear functions can be identified with the measurement of the sink 3 by using Algorithm~\ref{alg:path_graph} since the two paths between nodes 1 and 3 have different lengths.}
    \label{fig:triangle_graph}
\end{figure}
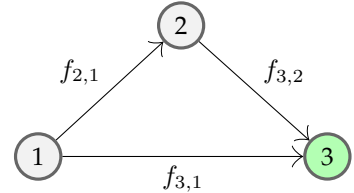 
When between two nodes $i$ and $j$ the length of all the paths is different, Algorithm~\ref{alg:path_graph} can still be used for the identification, since the information arriving from $i$ will affect $j$ at different orders of the derivatives depending on the path.
For instance, let us consider the DAG in Fig.~\ref{fig:triangle_graph}. Notice that the edge $f_{3,2}$ can be easily identified by setting different initial conditions to the node 2 and using the information of $\dot x_3$. Similarly, the edge $f_{3,1}$ can be identified by setting different initial conditions to the node 1 and using the information of $\dot x_3$, since the information coming through the node 2 will only  affect 3 with the second order derivative $\ddot x_3$. Finally, the identification of the edge $f_{2,1}$ implies the use of the second order derivative:
\begin{align*}
    \ddot x_3&=\frac{\partial f_{3,2}}{\partial x_{2}}\dot x_2+\frac{\partial f_{3,1}}{\partial x_{1}}\dot x_1\\
    &=\frac{\partial f_{3,2}}{\partial x_{2}}(f_{2,1}(x_1))+\frac{\partial f_{3,1}}{\partial x_{1}}u_1,
\end{align*}
where we have already identified $f_{3,2}$ and $f_{3,1}$, and we can set $\uu=0^T$. However, this procedure cannot be applied in a DAG where two nodes are connected by two or more paths of the same length, since the information is indistinguishable. For example, let  us consider again the DAG in Fig.~\ref{fig:model} where we assume that the edges $f_{4,2}$ and $f_{4,3}$ are linear and have already been identified. The second order derivative $\ddot x_4$ satisfies
\begin{equation*}
\ddot{x}_4 =  \frac{\partial f_{4,2}}{\partial x_2} \left( f_{2,1}(x_1) \right) + \frac{\partial f_{4,3}}{\partial x_3} \left( f_{3,1}(x_1) \right),
\end{equation*}
where the edges $f_{2,1}$ and $f_{3,1}$ cannot be distinguished, since both  affect the node 4 at the second derivative. Therefore, we need a different algorithm for the identification of the edges in parallel paths with the same length.

In the DAG in Fig. \ref{fig:model},  the information from the node 1 arrives to the measured node 4 through two different paths. The dynamics are given by 
\begin{align}
\label{eq:DAG_4nodes}
\begin{aligned}
\dot{x}_4 &= f_{4,2}(x_2) + f_{4,3}(x_3) + u_4 \\
\dot{x}_3 &= f_{3,1}(x_1) + u_3 \\ 
\dot{x}_2 &= f_{2,1}(x_1) + u_2 \\
\dot{x}_1 &= u_1.
\end{aligned}  
\end{align}
As in Assumption \ref{ass:dic_fun}, we assume that the unknown functions $f_{i,j}$ are linear combinations of known dictionary functions, i.e.,
\[
f_{i,j} = \sum_\ell {\alpha}^*_{ij\ell} \phi_{ij\ell},
\]
and we formulate the following algorithm to identify the coefficients $\alpha^*_{ij\ell}$.  
\begin{enumerate}[wide, labelwidth=!, labelindent=0pt]
    \item \emph{Identification of $f_{4,3}$ and $f_{4,2}$}: Setting $x_1(0) = 0$, $x_2(0) = 0$, $x_3(0) \neq 0$, $x_4(0)= 0$, and inputs $\uu = 0^T$, the network dynamics is equivalent to a simple path graph with 2 nodes. The coefficients $\alpha^*_{43\ell}$ are identified with the procedure for a path graph discussed in~Section~\ref{sec:Algorithm} by using the first-order derivatives of $x_4$. Similarly, we can identify $f_{4,2}$ by assigning $x_1(0) = 0,x_2(0) \neq 0, x_3(0) = 0, x_4(0) = 0$, and $\uu = 0^T$. 
    \item \emph{Identification of $f_{2,1}$ and $f_{3,1}$}: From the knowledge of $f_{4,3}$ and $f_{4,2}$, we proceed to compute the functions $f_{2,1}$ and $f_{3,1}$. Since both functions $f_{2,1}$ and $f_{3,1}$ depend on $x_1$, it is natural that we assume they share the same set of dictionary functions $\phi_{1\ell}$ over $x_1$ with coefficients $\alpha_{21\ell}$ and $\alpha_{31\ell}$ respectively. Given the dynamics \eqref{eq:DAG_4nodes}, the second-order derivative of $x_4$ can be written as 
    \begin{align}
        \label{eq:x4_ddot_Bridge}
        \begin{aligned}
        \ddot{x}_4 &=  \frac{\partial f_{4,2}}{\partial x_2} \dot{x}_2 + \frac{\partial f_{4,3}}{\partial x_3} \dot{x}_3 \\
        &=  \frac{\partial f_{4,2}}{\partial x_2} \left( f_{2,1}(x_1) \right) + \frac{\partial f_{4,3}}{\partial x_3} \left( f_{3,1}(x_1) \right)
        \end{aligned}
    \end{align}
    Initializing $x_1(0) \neq 0$, $x_2(0) \neq 0$ and $x_3(0) \neq 0$, we define 
    \begin{align}
    \label{eq:gamma_def}
    \gamma_{2} := \frac{\partial f_{4,2}}{\partial x_2}|_{x_{20}}  \quad \mbox{and} \quad \gamma_{3} := \frac{\partial f_{4,3}}{\partial x_3}|_{x_{30}}.
    \end{align}
    Measuring $\ddot{x}_4$ at $t=0$, we have 
    \begin{align*}
    \begin{aligned}
       \ddot{x}_4|_{t=0} &= \frac{\partial f_{4,2}}{\partial x_2}|_{x_{2}(0)} \left( \sum_\ell \alpha_{21\ell} \phi_{1\ell}(x_1(0)) \right) \\
       &\quad\; +  \frac{\partial f_{4,3}}{\partial x_3}|_{x_{3}(0)} \left( \sum_\ell \alpha_{31\ell} \phi_{1\ell}(x_1(0)) \right).
       \end{aligned}
       \end{align*}
      Substituting \eqref{eq:gamma_def} in the above equation, we get
      \begin{align*}
      \begin{aligned}
      \ddot{x}_4|_{t=0} & = \sum_{\ell} \gamma_2 \alpha_{21\ell} \phi_{1\ell}(x_1(0)) + \sum_{\ell} \gamma_3 \alpha_{31\ell} \phi_{1\ell}(x_1(0)) \\ 
      &= \sum_\ell \left( \gamma_{2} \alpha_{21\ell} +\gamma_{3} \alpha_{31\ell} \right) \phi_{1\ell}(x_1(0)). 
    \end{aligned}   
    \end{align*}
    From the initial conditions $x_{j}^{(k)}(0) \neq 0, j = 1,2,3$, $k = 1,\dots,K$ and the corresponding measurement of $\ddot{x}_4^{(k)}(0)$, we can solve for the coefficients $\alpha^*_{21\ell}$ and $\alpha^*_{31\ell}$ by formulating the regression problem
    \small{\begin{align}
    \label{eq:reg_DAG}
    \begin{aligned}
    &{\alpha}^*_{21\ell},{\alpha}^*_{31\ell} = \\ &\argmin_{\alpha_{21\ell},{\alpha}_{31\ell}} \! \sum_{k=1}^{K} \!\bigg(\!\ddot{x}_4^{(k)}(0) \!-\! \sum_{\ell} \!\left( \gamma_2^{(k)} \alpha_{21\ell} +\gamma_3^{(k)} \alpha_{31\ell} \right) \phi_{1\ell}(x_1^{(k)}(0))\bigg )^2,
    \end{aligned}
    \end{align}}
    \normalsize{}
    where 
    \[
    \gamma_2^{(k)} = \frac{\partial f_{4,2}}{\partial x_2}|_{x_{2}^{(k)}(0)} \quad  \quad \gamma_{3}^{(k)} = \frac{\partial f_{4,3}}{\partial x_3}|_{x_{3}^{(k)}(0)}. 
    \]
    Through proper choices of $x_j^{(k)}(0)$, $j = 1,2,3$ and $k = 1,\dots,K$ with $K$ greater than the total number of dictionary functions $\phi_{1\ell}$, the above regression problem is well-posed.
\end{enumerate}
The above procedure for the identification of the DAG in Fig.~\ref{fig:model} can be generalized to multiple parallel paths of length $2$ between two nodes. Consider a network with $n$ nodes, with nodes $2,3,\dots,n-1$ as the different intermediate (and parallel) nodes between nodes $1$ and $n$ with the dynamics 
\begin{align*}
    \begin{aligned}
     \dot{x}_1 &= u_1 \\
    \dot{x}_2 &= f_{2,1}(x_1)+u_2 \\   
    & \ \ \vdots \\ 
    \dot{x}_{n-1} &= f_{n-1,1}(x_1) +u_{n-1} \\ 
    \dot{x}_n &= \sum_{j=2}^{n-1}f_{n,j}(x_j) + u_n.
    \end{aligned}
\end{align*}
The algorithm for the identification of DAGs is summarized in Algorithm \ref{alg:bridge_graph}.
\begin{algorithm}
\caption{Identification with multiple parallel paths of length $2$}
\label{alg:sec_step_multiple}
\begin{algorithmic}[1] 
 \renewcommand{\algorithmicrequire}{\textbf{Input:}}
 \renewcommand{\algorithmicensure}{\textbf{Output:}}
 \REQUIRE Measurement of the sink (node $n$), set of dictionary functions $\phi_{nj\ell}, j =2,\dots,n-1,$ and $ \phi_{1\ell}$
 \ENSURE The parameters $\alpha^*_{nj\ell}$, and $\alpha^*_{j1\ell}$, where $j=2,\dots,n-1$, corresponding to the functions $f_{n,j}$ and $f_{j,1}$.
\STATE For $j=2,\dots,n-1$, set $x_{i}^{(k)}(0) = 0$ for $i\neq j$, $\uu = 0^T$ and choose $x_j^{(k)} \neq 0$ for $k=1,\dots,K$ and using Algorithm \ref{alg:path_graph}, identify the coefficients $\alpha^*_{nj\ell}$.  
\STATE Choose $x_{j}^{(k)}(0) \neq 0$ for $j=1,\dots,n-1$, $\uu = 0^T$ and from the knowledge of $f_{n,j}$, $j=2,\dots,n-1$ through coefficients computed in Step 1, define 
\[
\gamma_j^{(k)} := \frac{\partial f_{n,j}}{\partial x_j}|_{x_{j}^{(k)}(0)}.
\]
\STATE Retrieve the coefficients $\alpha^*_{j1\ell}$ by solving the regression problem  
\begin{align*}
&\ \ {\alpha}^*_{j1\ell} =\\
&\argmin_{\alpha_{j1\ell}} \!\sum_{k=1}^{K}\!\! \bigg(\!\ddot{x}_4^{(k)}(0) \!-\! \sum_{\ell} \!\left[ \small{\sum_{j=2}^{n-1}\gamma_i^{(k)} \alpha_{j1\ell}} \right] \phi_{1\ell}(x_1^{(k)}(0))\!\!\bigg)^2\!.
\end{align*}
\end{algorithmic}
\label{alg:bridge_graph}
 \end{algorithm}

 The procedure can be extended to networks having multiple parallel paths of equal length $n$ (and $n>2$) as follows. The last $n-1$ edges in each independent path can be identified using Algorithm \ref{alg:path_graph} by treating them as a path graph of length $n-1$, and the initial conditions of the nodes in the other paths are set to $0$. The $k^{th}$ edge from the sink is identified through the $k^{th}$-order derivative of the sink by applying Algorithm~\ref{alg:path_graph}. Then the procedure developed in Algorithm ~\ref{alg:bridge_graph} can be used to identify the edges connected to the source node.  

The combination of Algorithms~\ref{alg:path_graph} and \ref{alg:bridge_graph} allows us to ideally identify any DAG from nonzero initial conditions with only the measurement of the sinks. Unlike a classical nonlinear system identification, Algorithms~\ref{alg:path_graph} and \ref{alg:bridge_graph} are based only on the information obtained from the sinks, independently of other nodes of the network whose information is not available. Therefore, the choice of the algorithm for the identification of a subset of edge dynamics depends clearly on the network topology.
For instance, let us consider the DAG in Fig.~\ref{fig:bridge_graph_additional}. The edges $f_{4,1}$, $f_{4,2}$ and $f_{4,3}$ can be easily identified by using Algorithm~\ref{alg:path_graph}, and the edges $f_{2,1}$ and $f_{3,1}$ can be identified by using Algorithm~\ref{alg:bridge_graph} since they are the initial edges in the two paths 1-2-4 and 1-3-4 with the same length.

\begin{figure}[ht]
    \centering
    \begin{tikzpicture}
    [
roundnodes/.style={circle, draw=black!60, fill=black!5, very thick, minimum size=1mm},roundnode/.style={circle, draw=white!60, fill=white!5, very thick, minimum size=1mm},roundnodes2/.style={circle, draw=black!60, fill=green!30, very thick, minimum size=1mm}
]
\node[roundnodes](node1){1};
\node[roundnodes](node2)[above=of node1,yshift=1mm,xshift=1.9cm]{2};
\node[roundnodes2](node4)[right=3.2cm of node1]{4};
\node[roundnodes](node3)[below=of node1,yshift=1mm,xshift=1.9cm]{3};

\draw[red,-{Classical TikZ Rightarrow[length=1.5mm]}] (node1) to node [left,swap,yshift=2.5mm] {$f_{2,1}$} (node2);
\draw[blue,-{Classical TikZ Rightarrow[length=1.5mm]}] (node2) to node [right,swap,yshift=2.5mm] {$f_{4,2}$} (node4);
\draw[red,-{Classical TikZ Rightarrow[length=1.5mm]}] (node1) to node [left,swap,yshift=-2.5mm] {$f_{3,1}$} (node3);
\draw[blue,-{Classical TikZ Rightarrow[length=1.5mm]}] (node3) to node [right,swap,yshift=-2.5mm] {$f_{4,3}$} (node4);
\draw[blue,-{Classical TikZ Rightarrow[length=1.5mm]}] (node1) to node [above,swap,yshift=0mm] {$f_{4,1}$} (node4);

\end{tikzpicture}
    \caption{DAG whose nonlinear functions can be identified with the measurement of the sink 4 by using Algorithms~\ref{alg:path_graph} and \ref{alg:bridge_graph}. The identification of the edges $f_{4,1}$, $f_{4,2}$ and $f_{4,3}$ can be done by using Algorithm~\ref{alg:path_graph} while Algorithm~\ref{alg:bridge_graph} can be used for the identification of the edges $f_{2,1}$ and $f_{3,1}$.}
    \label{fig:bridge_graph_additional}
\end{figure}
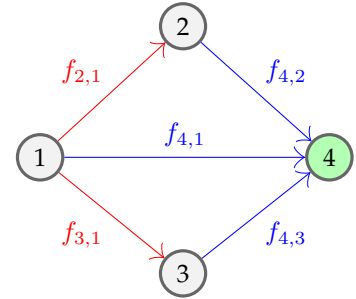

\subsection{Numerical Examples}
In this section, we provide numerical simulations to demonstrate the algorithms. In these examples, we assume that each measurement obtained by a node is corrupted by additive Gaussian noise with a known mean and standard deviation. Furthermore, the measurement noise is assumed to be independent across different nodes and across different time instants. The time-derivatives of the signals are estimated from these sampled node measurements by numerical differentiation.
\subsubsection{Numerical example for trees}

We first illustrate our network identification procedure with a $3$-node path graph. Consider the dynamics defined by
\begin{align}
\begin{aligned}
	\dot{x}_1 &= u_1 \\
    \dot{x}_2 &= -2x_1 + 1.25 x_1^2 + u_2 \\
    \dot{x}_3 &= -x_2 + 0.7x_2^2 -0.6x_2^3 + u_3.
\end{aligned}
\end{align}
The functions $f_{3,2}(x_2)$ and $f_{2,1}(x_1)$ are to be recovered from the data. The experiment is performed as discussed in Section~\ref{sec:Algorithm}. $50$-data paths are generated with nonzero initial conditions chosen uniformly randomly from $[-1,1]$ for each experiment. For each data path, the values of the sink (node $3$) are measured with a sampling time of $0.4$ for $10$ samples and the time-derivatives are estimated using a Savitzky-Golay filter \cite{SG1964} of window length $10$ and polynomial degree $5$. The measurement of the sink (node $x_3$) was of the order of $10^{-2}$, and hence it was corrupted by a Gaussian noise of mean $0$ and standard deviation $\sigma = 0.001$. Since we assume that the functions are analytic, we use a dictionary function of monomials $x_i,x_i^2,x_i^3,x_i^4$. The estimates of $f_{2,1}$ and $f_{3,2}$ from the data are

\vspace{-3mm}
\small{\begin{align}
{\begin{aligned}
	 	f_{2,1}(x_1) &= -2.004x_1  +1.253x_1^2 + 0.01x_1^3 + 0.003x_1^4 \\
        f_{3,2}(x_2) &=  -0.999x_2 +0.706x_2^2 -0.605x_2^3 + 0.011x_1^4.
\end{aligned}}
\end{align}}
\normalsize
The Root Mean Squared Error (RMSE) in the identification of all the edges is equal to $0.0022$.

A comparative study is performed for different standard deviation $\sigma$ of the measurement noise and it is observed that, as $\sigma$ decreases, the RMSE of the identification decreases (Table~\ref{tab1}).
\begin{table}[h]
	\centering
	\begin{tabular}{|c|c|}
		\hline
		$\bm{\sigma}$ & \textbf{RMSE} \\ \hline 
		$10^{-1}$ & $0.410$ \\ \hline
        $10^{-2}$ & $0.027$ \\ \hline
		$10^{-3}$ & $0.002$ \\ \hline
		$10^{-4}$& $0.001$  \\ \hline
	\end{tabular}
	\vspace{2mm}
	\caption{RMSE for various values of $\sigma$}
	\label{tab1}
\end{table}

\subsubsection{Numerical example for DAGs}

Now, we consider a 4-node DAG as in Fig. \ref{fig:model} with the functions given by
\begin{align}
    \label{eq:DAG_numex}
    \begin{aligned}
    f_{4,2}&= x_2 -1.25x_2^2 \\
    f_{4,3}&= 1.3x_3^2-0.6x_3^3 \\
    f_{2,1}&= -1.25x_1 -0.6x_1^3+ 0.2x_1^4 +  \sin(10x_1) \\ 
    f_{3,1}&= 1.3x_1^2-0.6x_1^3+ \sin(5x_1).
\end{aligned}
\end{align}
The dictionary functions $f_{i,j}$ are chosen as follows:
\begin{align*}
\mbox{for } f_{4,2}:\ & \{x_2,x_2^2,x_2^3,x_2^4\} \\
\mbox{for } f_{4,3}:\ & \{x_3,x_3^2,x_3^3,x_3^4\} \\
\mbox{for } f_{2,1} \mbox{ and } f_{3,1}:\ & \{x_1,x_1^2,x_1^3,x_1^4,\sin(5x_1),\sin(10x_1)\}. 
\end{align*}
The experiment is performed as discussed in Section~\ref{sec:Algorithm_DAG} for 100 data paths ($K = 100$) with nonzero initial conditions chosen uniformly randomly over $[-1,1]$. The value of node $4$ is measured with a sampling time of $0.3$, and $10$ samples are measured for each initial condition. The measurement of the node $x_4$ was of the order of $10^{-3}$, and hence it was corrupted with a Gaussian noise of mean $0$ and standard deviation $\sigma = 0.0001$. Similarly to the examples for trees, the time-derivatives of the node are computed using a Savitzky–Golay filter with window length $10$ and polynomial degree $5$. The identified functions are

\vspace{-3mm}
\small
\begin{align}
\begin{aligned}
   f_{4,2}(x_2)  &= 1.004x_2 -1.255x_2^2-0.00055 x_2^3-0.00892x_2^4 \\
   f_{4,3}(x_3) &= -0.00311x_3+1.305x_3^2 -0.605x_3^3 - 0.00595x_3^4 \\
   f_{2,1}(x_1) &= -1.252x_1 -0.00103x_1^2 -0.608x_1^3+0.202x_1^4 \\
        &\ \ \ \; -0.00049\sin(5x_1)+ 1.003\sin(10x_1) \\
   f_{3,1}(x_1) &= 0.00003 x_1+  1.307x_1^2
       -0.6004x_1^3 -0.0031x_1^4 \\ 
       &\ \ \ \; +1.004\sin(5x_1) +0.00029\sin(10x_1).
\end{aligned}
\end{align}

\normalsize
\noindent The RMSE in the estimation of the coefficients for different standard deviation $\sigma$ of the noise are given in Table~\ref{tab2}.
\begin{table}[h]
       \centering
    \begin{tabular}{|c|c|}
    \hline
         $\bm{\sigma}$ & \textbf{RMSE} \\ \hline 
         $10^{-2}$ & $0.126$ \\ \hline 
         $10^{-3}$ & $0.011$ \\ \hline 
         $10^{-4}$ & $0.006$ \\ \hline 
    \end{tabular}
    \vspace{2mm}
    \caption{RMSE for various values of $\sigma$}
    \label{tab2}
\end{table} 

The two examples illustrate the algorithms developed in this paper. The noise in the measurement poses a significant challenge in estimating the higher order derivatives. As mentioned earlier, identifying a function that lies $n$ nodes away from the sink requires computing its $n^{\text{th}}$-order derivative. Our algorithms do not address the estimation of the higher order derivatives, and this is a limitation of our approach. However, there exist well-known mathematical procedures that could be used to obtain this information. In the numerical studies, these derivatives are approximated using finite-difference methods, which typically require higher sampling rates. As an alternative, additional sensors could be used to directly measure signal derivatives. However, both the measurement and the estimation of derivatives effectively rely on high-pass filtering, which makes them highly sensitive to noise. In the numerical examples, it was observed that, when there is no noise in the time-series data and the sampling time is low, we were able to obtain good estimates of time-derivatives and, therefore, identify edges further away from the sink. 

Additionally, the multi-step nature of the identification process propagates the error through subsequent steps. When these propagated errors are combined with inaccuracies in the measurement of higher-order derivatives, they could limit the maximum number of edges that can be reliably identified from the sink. This implies that, in real applications, it could be necessary to measure additional nodes to obtain more accurate results like in the partial excitation and measurement case \cite{vizuete2024partial}.

\section{Conclusions and future work}\label{sec:conclusions}

We derived identifiability conditions for nonlinear edge functions in trees and general DAGs in the continuous-time setting. Similarly to the discrete-time case, we showed that the measurement of the sinks is necessary and sufficient for the identifiability of the networks by considering general nonlinear functions in the case of trees and pure nonlinear functions in the case of general DAGs. Based on these identifiability conditions, we proposed two algorithms for the identification of continuous-time networked dynamics based on the measurement of higher-order time derivatives. Finally, we showed that any DAG can be identified by using a combination of the two algorithms in the case of pure nonlinear functions.

 A natural extension of this work would be to derive identifiability conditions in the continuous-time setting for more complex network topologies, including cycles and loops \cite{vizuete2026nonlinear}. In this case, the main difficulty remains in the solution to the ODE \eqref{eq:model}, which depends on the solutions of other ODEs associated with in-neighbors and does not have an explicit form in most cases. An identification method could also be designed for the case of partial excitation 
and measurement \cite{vizuete2024partial}, and non-additive models \cite{vizuete2025path,vanelli2025local}. Finally, the estimation of  higher-order time derivatives might be problematic in certain scenarios, so the use of techniques such as the Koopman operator \cite{mauroy2020thekoopman,anantharaman2025koopman} could be an important research direction in this field.

\bibliographystyle{IEEEtran}
\bibliography{references}

\appendix

\subsection{Proof of Lemma~\ref{lemma:removal_node}}\label{app:proof}
In the DAG $G$, for constant excitation signals, the measurement of the node $i$ provides the output
\begin{align*}
x_i(t)&=F_i(\x^0,\uu,t)\\
    &=\int_{0}^t\sum_{\ell\in \NN_i}f_{i,\ell}(x_\ell(\tau))d\tau+u_it+x_i^0\\
    &=\sum_{\ell\in \NN_i}\prt{\int_{0}^t f_{i,\ell}(x_\ell(\tau))d\tau}+u_it+x_i^0.
\end{align*}
   Let us assume that there exists a set $\{\tilde f\}\neq \{ f\}$ such that $F_i=\tilde F_i$, which implies
   \begin{equation*}
\sum_{\ell\in \NN_i}\int_{0}^t f_{i,\ell}(x_\ell(\tau))d\tau=\sum_{\ell\in \NN_i}\int_{0}^t \tilde f_{i,\ell}(\tilde x_\ell(\tau))d\tau;    
\end{equation*}
\begin{multline}
    \sum_{\ell\in \NN_i}\int_{0}^t f_{i,\ell}(g_\ell(\x^0,\uu,\tau)+u_\ell \tau+x_\ell^0)d\tau=\\
    \sum_{\ell\in \NN_i}\int_{0}^t \tilde f_{i,\ell}(\tilde g_\ell(\x^0,\uu,\tau)+u_\ell \tau+x_\ell^0)d\tau,
\end{multline}
where we used $x_\ell(t)=g_\ell(\x^0,\uu,t)+u_\ell t+x_\ell^0$.
    Since the measurement of $j$ implies the knowledge of $g_j$ (i.e., $g_j=\tilde g_j$), and the edge $f_{i,j}$ is identifiable (i.e., $f_{i,j}=\tilde f_{i,j}$), we obtain   \begin{multline}\label{eq:equal_f_tilde_one_path}
        \sum_{\ell\in \NN_i\setminus\{j\}}\int_{0}^t f_{i,\ell}(g_\ell(\x^0,\uu,\tau)+u_\ell \tau+x_\ell^0)d\tau=\\
    \sum_{\ell\in \NN_i\setminus\{j\}}\int_{0}^t \tilde f_{i,\ell}(\tilde g_\ell(\x^0,\uu,\tau)+u_\ell \tau+x_\ell^0)d\tau.
    \end{multline}
    Next, we consider the induced subgraph $G_{V\setminus\{ j \}}$ where the node $j$ has been removed. The measurement of the node $i$ provides the output
    \begin{align}
    x_i(t)&=\sum_{\ell\in \NN_i\setminus\{j\}}\prt{\int_{0}^t f_{i,\ell}(x_\ell(\tau))d\tau}+u_it+x_i^0\nonumber\\
    &=\sum_{\ell\in \NN_i\setminus\{j\}}\prt{\int_{0}^t f_{i,\ell}(\hat g_\ell(\x^0,\uu,\tau)+u_\ell \tau+x_\ell^0)d\tau}\nonumber\\
    &\hspace{4mm}+u_it+x_i^0.\label{eq:node_i_subgraph}
    \end{align}
    Notice that the functions $\hat g_\ell$ do not depend on any information coming through $j$ since the node $j$ has only one path to $i$, so that $\hat g_\ell=g_\ell$, and \eqref{eq:node_i_subgraph} implies \eqref{eq:equal_f_tilde_one_path}. Therefore, if the edges $f_{i,\ell}$ and the functions $g_\ell$ for $\ell\in\NN_i\setminus\{j\}$ are identifiable in $G_{V\setminus\{ j \}}$, they are also identifiable in $G$.

\begin{IEEEbiography}[{\includegraphics[width=1in,height=1.25in,clip,keepaspectratio]{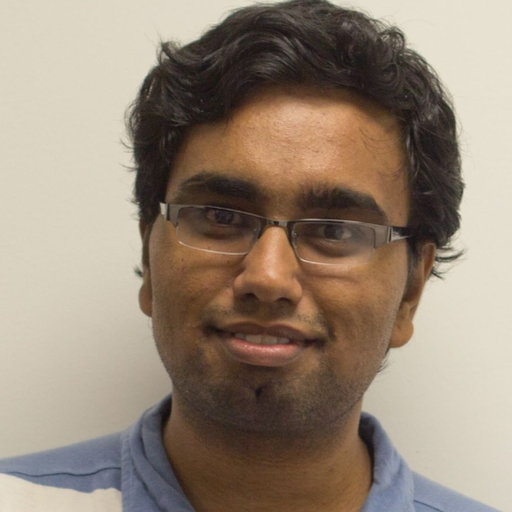}}]{Ramachandran Anantharaman} is  an Assistant Professor at Department of Electrical Engineering, Shiv Nadar University, Noida, India. He received his M.Tech (2012) and Ph.D (2021) degrees in Electrical Engineering from Indian Institute of Technology Bombay, India. He was a postdoctoral researcher with University of Namur, Belgium (2022-2024) and with Eindhoven University of Technology (TU/e), Netherlands (2024-2025). His research interests are in Boolean Control Networks, Data-driven models and control of nonlinear systems, Network identification, Control of switched mode power converters.
\end{IEEEbiography}

\begin{IEEEbiography}[{\includegraphics[width=1in,height=1.25in,clip,keepaspectratio]{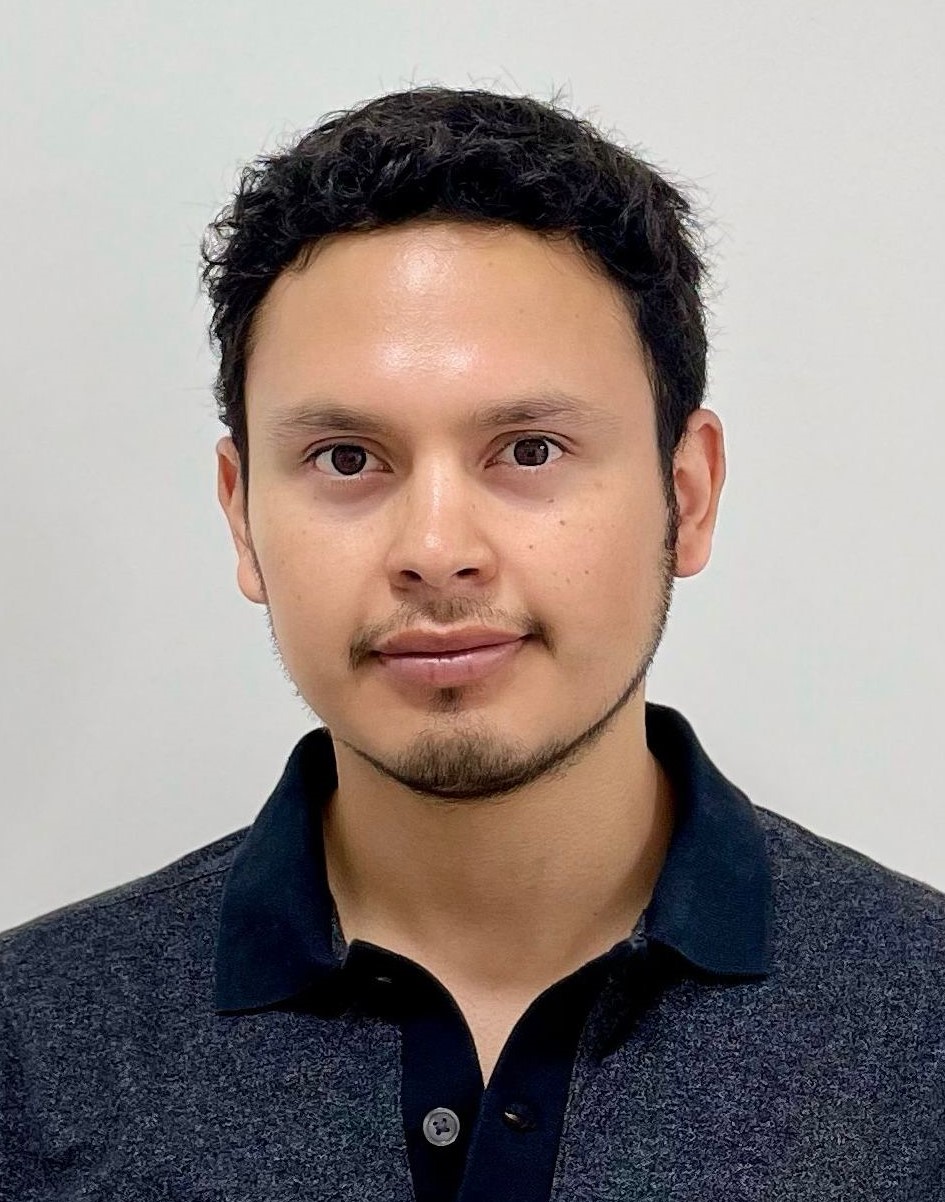}}]{Renato Vizuete} is an Associate Professor at Université de Lorraine, France. He
received the M.S. degree (très bien) in Systems, Control, and Information Technologies from Université Grenoble Alpes, France, in 2019, and the PhD degree in Automatic Control from Université Paris-Saclay, France, in 2022. He was a FNRS postdoctoral researcher (CR) at UCLouvain, Belgium, until 2026. His research interests include optimization, control theory, machine learning, network theory and system identification. He was the recipient of the Networks and Communication Systems TC Outstanding Student Paper Prize of the IEEE Control Systems Society in 2022, and the Second Thesis Prize in the category Impact Science of the Fondation CentraleSupélec in 2023.
\end{IEEEbiography}

\begin{IEEEbiography}[{\includegraphics[width=1in,height=1.25in,clip,keepaspectratio]{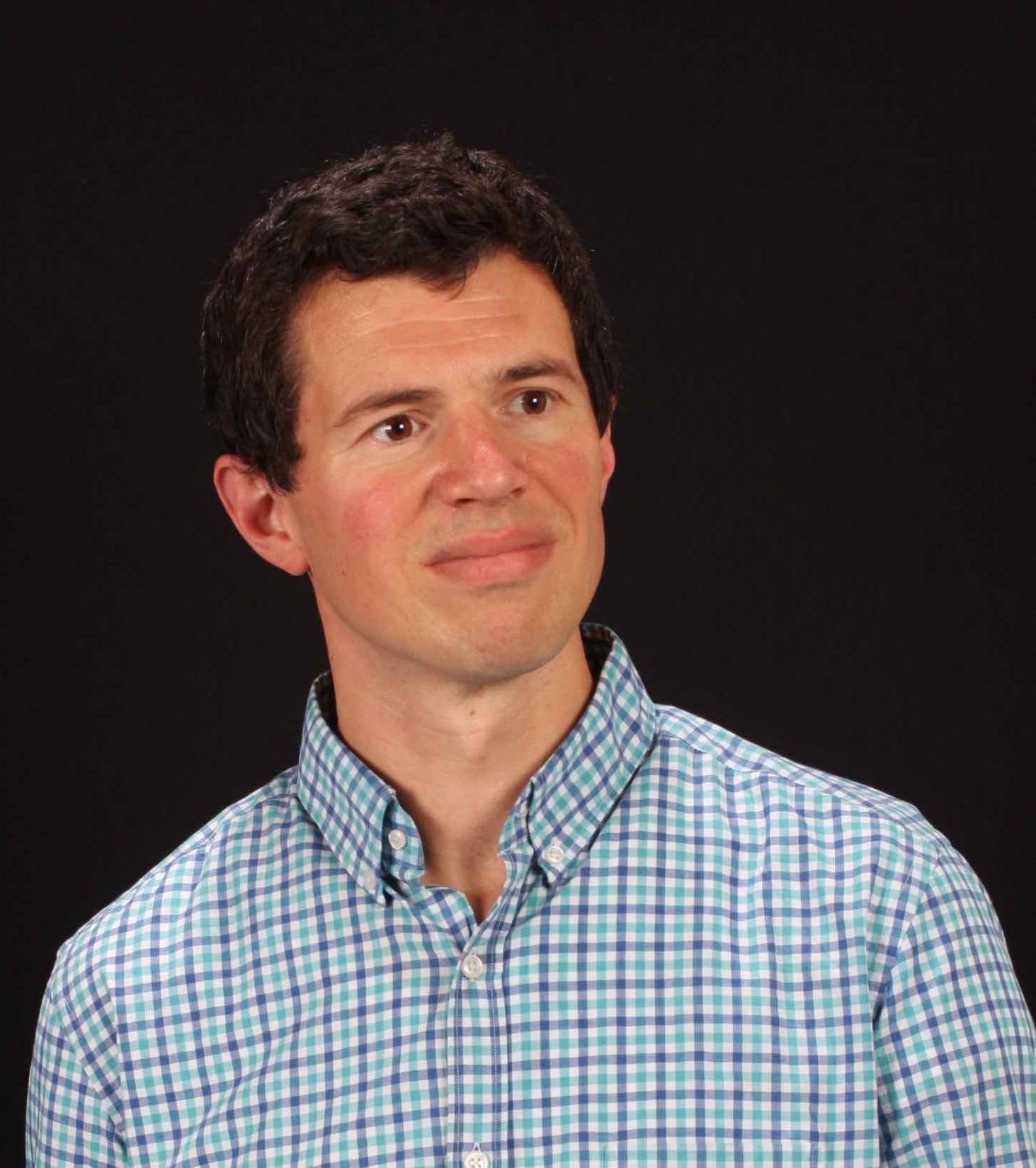}}]{Julien M. Hendrickx} is professor of mathematical engineering at UCLouvain, in the Ecole Polytechnique de Louvain since 2010, and is currently the head of the ICTEAM institute.

He obtained an engineering degree in applied mathematics (2004) and a PhD in mathematical engineering (2008) from the same university. He has been a visiting researcher at the University of Illinois at Urbana Champaign in 2003-2004, at the National ICT Australia in 2005 and 2006, and at the Massachusetts Institute of Technology in 2006 and 2008. He was a postdoctoral fellow at the Laboratory for Information and Decision Systems of the Massachusetts Institute of Technology 2009 and 2010, holding postdoctoral fellowships of the F.R.S.-FNRS (Fund for Scientific Research) and of Belgian American Education Foundation. He was also resident scholar at the Center for Information and Systems Engineering (Boston University) in 2018-2019, holding a WBI.World excellence fellowship.

Doctor Hendrickx is the recipient of the 2008 EECI award for the best PhD thesis in Europe in the field of Embedded and Networked Control, and of the Alcatel-Lucent-Bell 2009 award for a PhD thesis on original new concepts or application in the domain of information or communication technologies.
\end{IEEEbiography}

\begin{IEEEbiography}[{\includegraphics[width=1in,height=1.25in,clip,keepaspectratio]{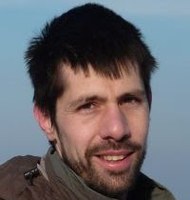}}]{Alexandre Mauroy} is an Associate Professor with the Department of Mathematics,  University of Namur, Belgium, and currently Director of the Namur Institute for Complex Systems (naXys).

He received the M.S. degree in aerospace engineering and the Ph.D. degree in applied mathematics from the University of Liège, Liège, Belgium, in 2007 and 2011, respectively. Prior to joining the University of Namur, he was a postdoctoral researcher with the University of California Santa Barbara from 2011 to 2013, the University of Liège from 2013 to 2015, and the University of Luxembourg in 2016. His research interests include synchronization in complex networks, network identification, and applications of operator-theoretic methods to dynamical systems and control.
\end{IEEEbiography}

\end{document}